\documentclass[a4paper,12pt]{article}
\usepackage{amsmath,amsthm,amsfonts}
\newtheorem{theorem}{Theorem}[section]

\newtheorem{corollary}[theorem]{Corollary}
\theoremstyle{definition}

\newtheorem{example}{Example}

\DeclareMathOperator{\GF}{GF}
\DeclareMathOperator{\PG}{PG}

\newcommand{\ZZ}{\mathbb{Z}}

\newcommand{\as}{(X,\{R_i\}_{i=0}^d)}
\newcommand{\fused}{(X,\{R_{\Lambda_j}\}_{j=0}^{d'})}

\newcommand{\allone}[1]{\mathbf{1}_{#1}}
\newcommand{\cD}{\mathcal{D}}
\newcommand{\cS}{\mathcal{S}}
\newcommand{\cX}{\mathcal{X}}
\begin{document}
%%%%%%%%%%%%%%%%%%%%%%%%%%%%

{\LARGE {\bf 
\noindent
Pseudocyclic Association Schemes and Strongly Regular Graphs}
}

\bigskip
\noindent
Version of April 16, 2009

\bigskip\bigskip\noindent
{\scshape Takuya Ikuta\footnote{permanent address: 
Faculty of Law, Kobe Gakuin University,
Minatojima, Chuo-ku, Kobe, 650-8586 Japan} 
} \\
{\small \em 
Graduate School of Information Sciences, 
Tohoku University,
Aramaki-Aza-Aoba, Aoba-ku, Sendai, 980-8579 Japan \\
\verb+ikuta@law.kobegakuin.ac.jp+
}

\bigskip\noindent
{\scshape Akihiro Munemasa} \\
{\small\em
Graduate School of Information Sciences, 
Tohoku University,
Aramaki-Aza-Aoba, Aoba-ku, Sendai, 980-8579 Japan \\
\verb+munemasa@math.is.tohoku.ac.jp+
}

\bigskip
{\small
\bigskip\noindent
{\bf Abstract.}
Let $\cX$ be a pseudocyclic 
association scheme in which all the nontrivial relations are
strongly regular graphs with the same eigenvalues. 
We prove that the principal part of the first eigenmatrix
of $\cX$
is a linear combination of an incidence matrix of a
symmetric design and the all-ones matrix.
Amorphous pseudocyclic association schemes are examples
of such association schemes whose associated symmetric design
is trivial.
We present several non-amorphous examples, which
are either cyclotomic association schemes, or their
fusion schemes.
Special properties of 
symmetric designs guarantee the existence of further
fusions, and the two known 
non-amorphous association schemes of class $4$ discovered
by van Dam and by the authors, 
are recovered in this way.
We also give another pseudocyclic non-amorphous 
association scheme of class $7$ on $\GF(2^{21})$, and
a new pseudocyclic amorphous association scheme of class $5$
on $\GF(2^{12})$.
}


%%%%%%%%%%%%%%%%
% Introduction
%%%%%%%%%%%%%%%%

\section{Introduction}

A.V.~Ivanov's conjecture \cite{IP}, 
though disproved by E.R.~van Dam,
asserted that, if each nontrivial relation 
in a symmetric association scheme is
strongly regular, then an arbitrary partition of the set of
nontrivial relations gives rise to an association scheme.
Association schemes
satisfying this conclusion is called \emph{amorphous} 
(or \emph{amorphic}). 
A counterexample to A.V.~Ivanov's conjecture was given by
van Dam in \cite{D0} for the imprimitive case, and in
\cite{D} for the primitive case. Presently 
there are only a few primitive counterexamples known. 
%They are both class $4$ association schemes, neither of which
%is pseudocyclic. 
%The first one is due to 
An example due to van Dam \cite{D} has the
first eigenmatrix given by
\begin{equation}  \label{e=45}
\begin{bmatrix}
1 & 3276 & 273 & 273 & 273 \\
1 & -52 & 17 & 17 & 17 \\
1 & 12 & -15 & -15 & 17 \\
1 & 12 & -15 & 17 & -15 \\
1 & 12 & 17 & -15 & -15
\end{bmatrix},
\end{equation}
and another one is due to the authors \cite{IM} with
the first eigenmatrix given by
\begin{equation}  \label{e=75}
\begin{bmatrix}
1 & 838860 & 69905 & 69905 & 69905 \\
1 & -820 & 273 & 273 & 273 \\
1 & 204 & -239 & -239 & 273 \\
1 & 204 & -239 & 273 & -239 \\
1 & 204 & 273 & -239 & -239
\end{bmatrix}.
\end{equation}

A symmetric association scheme $\cX=\as$ is said to be
pseudocyclic if
%a symmetric association scheme of class $d$ on $X$. If 
the nontrivial multiplicities $m_1,\dots,m_d$ of $\cX$
coincide.
%then $\cX$ is said to be \emph{pseudocyclic}. 
The first eigenmatrix of a pseudocyclic association scheme 
is of the form
\[
P=\begin{bmatrix}
1&f&\dots&f\\
1&&&\\
\vdots&&P_0&\\
1&&&
\end{bmatrix},
\]
where $f$ denotes the common nontrivial multiplicities, as well
as the common nontrivial valencies.
The submatrix $P_0$ is called the principal part of $P$.
If we restrict A.V.~Ivanov's conjecture to the
pseudocyclic case, it asserts that 
for pseudocyclic association scheme in which 
each nontrivial relation is
strongly regular, the principal part of its
first eigenmatrix is a linear combination of $I$ and
$J$, after a suitable permutation of rows.

Cyclotomic association schemes in which 
each nontrivial relation is
strongly regular, have been investigated in its own right.
It follows from McEliece's theorem (\cite{Mc}, see also
\cite[Lemma~2.8]{SW}) that the number of nontrivial eigenvalues
of the cyclotomic association scheme of class $e$ over $\GF(p^m)$
is the same as that of weights in the irreducible cyclic code $c(p,m,e)$
(see \cite[Definition~2.2]{SW}).
In this sense, such cyclotomic association schemes correspond to
two-weight irreducible cyclic codes.
Moreover, under this correspondence,
subfield codes, semiprimitive codes correspond to amorphous
cyclotomic association scheme which are imprimitive, primitive, respectively.
Primitive amorphous cyclotomic association schemes
%Cyclotomic association schemes having this property
were investigated by Baumert, Mills and Ward \cite{BMW},
and Brouwer, Wilson and Xiang \cite{BWX}.
%They gave a necessary and sufficient
%condition for a cyclotomic scheme to be amorphous. 
%Although quite rare,
%there exist non-amorphous cyclotomic association schemes
%in which every nontrivial relation is strongly regular.
%These are given in Section~\ref{sec:main}.
%Therefore, non-amorphous 
Non-amorphous cyclotomic association schemes
in which every nontrivial relation is strongly regular,
are thus equivalent to exceptional two weight irreducible cyclic codes
in the sense of Schmidt and White \cite{SW}.
Therefore, cyclotomic association schemes corresponding to
exceptional two weight irreducible cyclic codes are
pseudocyclic counterexamples to A.V.~Ivanov's conjecture,
and there are eleven such codes in \cite{SW}.

One of the purpose of this paper is to 
show that both of the counterexamples with first eigenmatrices
(\ref{e=45}), (\ref{e=75}) are derived from some
pseudocyclic association schemes $\cX_1,\cX_2$, respectively,
of class $15$
which are also counterexamples themselves.
It turns out that, the principal part of the
first eigenmatrix of $\cX_1$ or $\cX_2$
is expressed
by an incidence matrix of $\PG(3,2)$.
In a more general setting, we prove in Theorem~\ref{main}
that the principal part of the first eigenmatrix is a linear
combination of an incidence matrix of a symmetric design
and the all-ones matrix.
For an amorphous pseudocyclic association scheme of class $d$,
the associated symmetric design is the complete
$2$-$(d,d-1,d-2)$ design. 
When the associated symmetric design is a projective space,
we show in Theorem~\ref{thm:proj} that
the existence of certain fusion schemes follows from
special properties of projective spaces.
This gives an explanation for the existence of the fusion 
schemes of class $4$ in $\cX_1,\cX_2$.
Moreover, the two pseudocyclic association schemes $\cX_1,\cX_2$
of class $15$ give rise to two pseudocyclic amorphous
fusion schemes of class $5$.
%, at least one of which is new.
We also give a pseudocyclic class $7$ fusion scheme of 
the cyclotomic association scheme of class $49$ 
on $\GF(2^{21})$.
Its associated design is $\PG(2,2)$.

\section{Preliminaries}

We refer the reader to \cite{BI} for notation and general theory of 
association schemes.
Let $\cX=\as$ be 
a symmetric association scheme of class $d$ on $X$. 
Let $P$ be the first eigenmatrix of $\cX$.
Let $\{\Lambda_j\}_{j=0}^{d'}$ be a partition of 
$\{0,1,\ldots,d\}$ with $\Lambda_0=\{0\}$, 
and we set $R_{\Lambda_j}=\bigcup_{\ell \in \Lambda_j} R_{\ell}$. 
If $\fused$ forms an association scheme, 
then we call $\fused$ a {\em fusion} scheme of $\cX$. 
If $\fused$ is an association scheme
for any partition $\{\Lambda_j\}_{j=0}^{d'}$ of $\{0,1,\ldots,d\}$ 
with $\Lambda_0=\{0\}$, 
then $\cX$ is called {\em amorphous}.
We refer the reader to a recent
article \cite{DM} for details on amorphous association schemes.

There is a simple criterion in terms of $P$ 
for a given partition $\{\Lambda_j\}_{j=0}^{d'}$ 
to give rise to a fusion scheme 
(due to Bannai \cite{B}, Muzychuk \cite{M}): 
There exists a partition 
$\{\Delta_i\}_{i=0}^{d'}$ of $\{0,1,\ldots,d\}$ 
with $\Delta_0=\{0\}$ such that each 
$(\Delta_i, \Lambda_j)$-block of the first eigenmatrix
$P$ has a constant row sum. 
The constant row sum turns out to be the $(i,j)$ entry 
of the first eigenmatrix of the fusion scheme.

An association scheme $\cX$ of class $d$
having the nontrivial multiplicities 
$m_1=\ldots =m_d$ is called {\em pseudocyclic}.
It is known that, in a pseudocyclic association scheme $\cX$,
all the nontrivial valencies coincide
(see \cite[p.76]{BI},or \cite[Proposition~2.2.7]{BCN}).
By the {\em principal part} of the first eigenmatrix, we mean
the lower-right $d \times d$ submatrix of 
the first eigenmatrix.

Let $q$ be a prime power and let $e$ be a divisor of $q-1$. 
Fix a primitive element $\alpha$ of the multiplicative group 
of the finite field $\GF(q)$. Then $\langle \alpha^e \rangle$ is 
a subgroup of index $e$ and its cosets are 
$\alpha^i \langle \alpha^e \rangle$ ($0 \leq i \leq e-1$). 
We define $R_0=\{(x,x) \mid x \in \GF(q)\}$ and 
$R_i=\{(x,y) \mid x-y \in \alpha^i \langle \alpha^e \rangle, \;
x,y \in \GF(q)\}$  ($1 \leq i \leq e$). 
Then $(\GF(q),\{R_i\}_{i=0}^e)$ forms an 
association scheme and is called the {\em cyclotomic 
association scheme}, or {\em cyclotomic scheme}, for short,
of class $e$ on $\GF(q)$.
A cyclotomic scheme is a pseudocyclic association 
scheme.

Suppose $q=p^m$, where $p$ is a prime.
The cyclotomic scheme of class $e$ on $\GF(q)$ is
amorphous if and only if $m$ is even and $e$ divides
$p^{m'}+1$ for some divisor $m'$ of $m/2$.
This is essentially due to Baumert, Mills and Ward \cite{BMW},
but see also \cite{BWX}.

\section{A symmetric design in the first eigenmatrix}
\label{sec:main}

\begin{theorem}  \label{main}
Let $(X,\{R_i\}_{i=0}^d)$ be a pseudocyclic association scheme 
of class $d$.
Assume that the graphs
$(X,R_i)$ ($i=1,\ldots,d$) 
are all strongly regular with the same eigenvalues.
Then there exists a symmetric $2$-$(d,k,\lambda)$ design 
$\cD$ such that the principal part of the first eigenmatrix
of $\cX$ is given by $rM+s(J-M)$, where $M$ is an incidence matrix
of $\cD$, $r$ and $s$ are the nontrivial eigenvalues of the graphs
$(X,R_i)$.
\end{theorem}
\begin{proof}
By the assumption, the principal part $P_0$ can be expressed as
$P_0=rM+s(J-M)$ for some $(0,1)$-matrix $M$. 
Then by the orthogonality relations
(see \cite[Chapter II, (3.10)]{BI}),
we find
\[
P_0J=-J,\quad fJ+P_0P_0^T=|X|I,
\]
where $f$ denotes the common nontrivial multiplicities.
The former implies 
\[
MJ=-\frac{sd+1}{r-s}J,
\]
hence $k=-(sd+1)/(r-s)$ is a positive integer. The latter
implies
\[
%MM^T=\frac{1}{(r-s)^2}(|X|I+(s^2|X|+2s-f)J).
MM^T=\frac{1}{(r-s)^2}(|X|I+(s^2d+2s-f)J).
\]
This implies that $M$ is an incidence matrix of a symmetric
design on $d$ points with block size $k$.
\end{proof}

The assumption that
the eigenvalues of the strongly regular graphs
appearing as the nontrivial relations are the same,
seems redundant. We have verified that the conclusion
of Theorem~\ref{main} holds without this assumption for $d\leq4$.

Next we show the existence of further fusions. We denote by
$\allone{n}$ the column vector of length
$n$ whose entries are all $1$.

\begin{corollary}\label{main:c}
Under the same assumptions as in Theorem~\ref{main},
$\cX$ has a fusion scheme of class $3$ with the first eigenmatrix
\begin{equation}\label{eq:class3}
\begin{bmatrix}
1&f&(k-1)f&(d-k)f\\
1&r&(k-1)r&(d-k)s\\
1&r&(\lambda-1)r+(k-\lambda)s&(k-\lambda)r+(d-2k+\lambda)s\\
1&s&\lambda r+(k-1-\lambda)s&(k-\lambda)r+(d-2k+\lambda)s
\end{bmatrix}.
\end{equation}
In particular,
there exists a fusion scheme of class $2$
with the first eigenmatrix
\begin{equation}\label{eq:class2}
\begin{bmatrix}
1&kf&(d-k)f\\
1&kr&(d-k)s\\
1&\lambda r+(k-\lambda)s&(k-\lambda)r+(d-2k+\lambda)s
\end{bmatrix}.
\end{equation}
\end{corollary}
\begin{proof}
Let $M$ be an incidence matrix of the design $\cD$, so that
$P_0=rM+s(J-M)$ holds. Without loss of generality, we may 
assume that the first $k$ columns of $M$ correspond to the
set of points on a block $B$ of $\cD$,
and that $B$ is represented by the first row of $M$.
Let $F$ denote the
$d\times 3$ matrix defined by
\[
F=\begin{bmatrix}
1&0&0\\
0&\allone{k-1}&0\\
0&0&\allone{d-k}
\end{bmatrix}.
\]
Then we have
\[
MF=\begin{bmatrix}
1&k-1&0\\
\allone{k-1}&(\lambda-1)\allone{k-1}&(k-\lambda)\allone{k-1}\\
0&\lambda\allone{d-k}&(k-\lambda)\allone{d-k}
\end{bmatrix}.
\]
It follows that the matrix $P_0F$ has $3$ distinct rows, which
are precisely those of the $3\times 3$ lower-right submatrix
of (\ref{eq:class3}). By the Bannai--Muzychuk criterion, we
obtain a fusion scheme of class $3$ with the first eigenmatrix
given by (\ref{eq:class3}).
Fusing the first two relations of this class $3$ association
scheme, we obtain a class $2$ association scheme with the
first eigenmatrix given by (\ref{eq:class2}).
\end{proof}

%Baumert, Mills, and Ward \cite{BMW}
%determined amorphous cyclotomic
Amorphous pseudocyclic association schemes 
%(see also \cite{BWX}). 
%These schemes 
satisfy the
conditions of Theorem~\ref{main}. However, the symmetric
design appearing in the principal part of the first eigenmatrix
is the complete
$2$-$(d,d-1,d-2)$ design. The conclusion of
Corollary~\ref{main:c} is trivially true for amorphous
association schemes. The nontrivial part of 
Corollary~\ref{main:c} is that it holds also for
non-amorphous association schemes.

Examples of 
%We have found only five non-amorphous 
cyclotomic
schemes satisfying the 
conditions of Theorem~\ref{main} have been investigated
thoroughly by Schmidt and White \cite{SW}, 
and some of the exceptional examples
were already found by Langevin \cite{Lan}.
The smallest example in \cite[Table 1]{SW}
is the cyclotomic scheme of class $11$ on $\GF(3^5)$,
which gives a unique symmetric $2$-$(11,5,2)$ design
by Theorem~\ref{main}.
%Theorem~\ref{main}. They are summarized
%in Table~\ref{t1}.
%\begin{table}[ht]
%\begin{center}
%\begin{tabular}{|c|c|c|}
%\hline
%field & design & eigenvalues\\
%\hline
%$\GF(3^5)$ & $(11,5,2)$ & $22,\;4,\;-5$ \\
%$\GF(3^{12})$ & $(35,17,8)$ & $15184,\;118,\;-125$ \\
%$\GF(5^9)$ & $(19,9,4)$ & $1953125,\;296,\;-329$ \\
%$\GF(7^9)$ & $(37,9,2)$ & $1090638,\;584,\;-1817$ \\
%$\GF(11^7)$ & $(43,21,10)$ & $453190,\;650,\;-681$ \\
%\hline
%\end{tabular}
%\end{center}
%\caption{Cyclotomic Examples}
%\label{t1}
%\end{table}
Its associated 
strongly regular graph is the coset graph of the
ternary Golay code (see \cite{BLS}), which was later
recognized as the cyclotomic graph by
van Lint and Schrijver \cite{LS} in 1981.
In this sense, a counterexample to A.V.~Ivanov's conjecture
\cite{IP} could be considered
known before the conjecture was announced in 1991.
We note that the fusion schemes of this cyclotomic
association scheme obtained by Corollary~\ref{main:c}
were already pointed out by Delsarte 
\cite[Example 2 on p.93]{Del}, in 1973.

%According to our computer search,
%the five examples given in Table~\ref{t1}
%are the only non-amorphous
%cyclotomic association schemes with fewer than
%$10^8$ points, satisfying the
%conditions of Theorem~\ref{main} .
%We do not know whether there are any other
%non-amorphous pseudocyclic association schemes
%satisfying the conditions of Theorem~\ref{main}.

There are three more pseudocyclic association schemes
satisfying the conditions of Theorem~\ref{main},
which are not cyclotomic schemes, but
fusions of cyclotomic schemes.
They will be given in the next section.

\section{Projective spaces and fusion schemes}
\label{sec:proj}
Let $q$ be a prime power, $m$ an integer greater than
$1$. By $\PG(m,q)$ we mean the symmetric $2$-$(d,k,\lambda)$
design consisting of the points and hyperplanes of
the projective space $\PG(m,q)$ of dimension $m$ over
$\GF(q)$, where $d=(q^{m+1}-1)/(q-1)$, 
$k=(q^m-1)/(q-1)$, $\lambda=(q^{m-1}-1)/(q-1)$.
Let $M$ be the hyperplane-point incidence matrix
of $\PG(m,q)$, and 
suppose that the columns of $M$ are indexed by the
points of $\PG(m,q)$ in such a way that
the last $q+1$ columns
correspond to the set of points on a line
$L=\{\beta_1,\ldots,\beta_{q+1}\}$.
Consider the following $d\times(q+2)$ matrix 
\[
F_1=\begin{bmatrix}
\allone{d-q-1} & 0\\
0& I_{q+1}
\end{bmatrix},
\]
If the rows of $M$ are indexed by $\lambda$ hyperplanes containing $L$,
$k-\lambda$ hyperplanes which meet $L$ at $\beta_1$,
$k-\lambda$ hyperplanes which meet $L$ at $\beta_2$, and so on,
then we have
\begin{equation}\label{eq:F1}
MF_1=\begin{bmatrix}
(k-q-1)\allone{\lambda}&&J_{\lambda\times(q+1)}&\\
(k-1)\allone{k-\lambda}&\allone{k-\lambda}&&0\\
\vdots&&\ddots&&\\
(k-1)\allone{k-\lambda}&0&&\allone{k-\lambda}
\end{bmatrix},
\end{equation}

A spread of $\PG(3,q)$ is a set of lines which
partition the set of points. A spread in
$\PG(3,q)$ exists for any prime power $q$.
Let $\cS=\{L_1,\ldots, L_{q^2+1}\}$ be a spread in $\PG(3,q)$. 
Let $M$ be the plane-point incidence matrix of $\PG(3,q)$,
and suppose that the columns of $M$ are indexed in
accordance with the partition $\cS$ of the points
of $\PG(3,q)$. 
Consider the following $(q^2+1)(q+1)\times(q^2+1)$ matrix 
\[
F_2=\begin{bmatrix}
\allone{q+1} && 0\\
&\ddots&\\
0&&\allone{q+1}
\end{bmatrix}.
\]
If the rows of $M$ are indexed by $q+1$ planes containing $L_1$,
$q+1$ planes containing $L_2$, and so on,
then we have
\begin{equation}\label{eq:F2}
MF_2=
\begin{bmatrix}
(q+1)\allone{q+1} &&\allone{q+1}\\
&\ddots&\\
\allone{q+1}&&(q+1)\allone{q+1}
\end{bmatrix}.
\end{equation}
%after a suitable permutation of rows.

\begin{theorem}\label{thm:proj}
Let $\cX$ be an association scheme of class
$d=(q^{m+1}-1)/(q-1)$ with the first eigenmatrix
\[
P=\begin{bmatrix}
1&f\allone{d}^T\\
\allone{d}&r M+s (J-M)
\end{bmatrix},
\]
where $M$ is an incidence matrix of $\PG(m,q)$. 
Let $k=(q^{m}-1)/(q-1)$.
Then the following statements hold.
\begin{enumerate}
\item[{\rm(i)}] There exists a fusion scheme of class $q+2$ with
the first eigenmatrix
\begin{equation}\label{P1}
\begin{bmatrix}
1&(d-q-1)f&f\allone{q+1}^T\\
1&(k-q-1)r+(d-k)s&r\allone{q+1}^T\\ 
\allone{q+1}&((k-1)r+(d-k-q)s)\allone{q+1}&(r -s ) I+s J
\end{bmatrix}.
\end{equation}
\item[{\rm(ii)}]
If $m=3$, then there exists an amorphous
fusion scheme of class $q^2+1$ with the first eigenmatrix
\[
P=\begin{bmatrix}
1&(q+1)f\allone{q^2+1}^T\\
\allone{q^2+1}&q(r   -s  ) I+(r +s  q) J
\end{bmatrix}.
\]
\end{enumerate}
\end{theorem}
\begin{proof}
(i) 
We can see easily from (\ref{eq:F1}) that the matrix
$(r  M+s (J-M))F_1$
has $q+2$ distinct rows, which are precisely those of
the lower-right $(q+2)\times (q+2)$ submatrix of
(\ref{P1}). Then the result follows from the
Bannai--Muzychuk criterion.

(ii) 
The proof is similar to (i), noting that the matrix
$(r  M+s (J-M))F_2$
has $q^2+1$ distinct rows.
\end{proof}

\begin{example}
Let $\alpha$ be an arbitrary
primitive element of $\GF(2^{12})$, and
%satisfying 
%\[
%\alpha^{12}+\alpha^6+\alpha^4+\alpha+1=0.
%\]
let
\[
H_j=\{(x,y)\mid x-y\in \alpha^j\langle \alpha^{45}\rangle\} \quad
 (j\in\ZZ).
\]
For a fixed integer $a$ which is relatively prime to $15$, we put
\[
R_k=\bigcup_{i=0}^2 H_{a(3(k-1)+5i)}.
\]
By computer, we have verified that the graph $\Gamma_k$ 
on $\GF(2^{12})$ with edge set $R_k$
is a strongly regular graph with eigenvalues 
$273,17,-15$, for each $k\in\{1,\ldots,15\}$.
In fact, these graphs are one of the strongly regular graphs
discovered by de Lange \cite{L}.
Together with the diagonal relation $R_0$,
we obtain a $15$-class pseudocyclic association scheme
$(\GF(2^{12}),\{R_i\}_{i=0}^{15})$
satisfying the hypothesis of Theorem~\ref{main}.
The principal part of the first eigenmatrix is a linear combination
of the all-ones matrix and an incidence matrix of a 
symmetric $2$-$(15,7,3)$ design.
Since this matrix is circulant by the definition of $R_k$,
the design is isomorphic to $\PG(3,2)$ by \cite[p.~984]{Hall}.

By Theorem~\ref{thm:proj}(i), we obtain a $4$-class fusion scheme
with the first eigenmatrix given by (\ref{e=45}).
By Theorem~\ref{thm:proj}(ii), we obtain a $5$-class pseudocyclic
amorphous association scheme. We have verified by
computer that this amorphous association scheme
is not isomorphic to the amorphous cyclotomic association
scheme of class $5$ on $\GF(2^{12})$.
\end{example}

\begin{example}
Let $\alpha$ be an arbitrary
primitive element of $\GF(2^{20})$, and let
\[
H_j=\{(x,y)\mid x-y\in \alpha^j\langle \alpha^{75}\rangle\} \quad
 (j\in\ZZ).
\]
For a fixed integer $a$ which is relatively prime to $15$, we put
\[
R_k=\bigcup_{i=0}^4 H_{a(5(k-1)+3i)}.
\]
By computer, we have verified that the graph $\Gamma_k$ 
on $\GF(2^{20})$ with edge set $R_k$
is a strongly regular graph with eigenvalues 
$69905,273,-239$, for each $k\in\{1,\ldots,15\}$.
Together with the diagonal relation $R_0$,
we obtain a $15$-class pseudocyclic association 
scheme $(\GF(2^{20}),\{R_i\}_{i=0}^{15})$
satisfying the hypothesis of Theorem~\ref{main}.
The principal part of the first eigenmatrix is a linear combination
of the all-ones matrix and an incidence matrix of a $2$-$(15,7,3)$ design.
Since this matrix is circulant by the definition of $R_k$,
the design is isomorphic to $\PG(3,2)$ by \cite[p.~984]{Hall}.

By Theorem~\ref{thm:proj}(i), we obtain a $4$-class fusion scheme
with the first eigenmatrix given by (\ref{e=75}).
By Theorem~\ref{thm:proj}(ii), we obtain a $5$-class pseudocyclic
amorphous association scheme. We conjecture that
the strongly regular graphs in this association scheme
are not isomorphic to a cyclotomic strongly regular graph,
and hence our amorphous association scheme is new.

\end{example}

\begin{example}
Let $\alpha$ be an arbitrary
primitive element of $\GF(2^{21})$, and let
\[
H_j=\{(x,y)\mid x-y\in \alpha^j\langle \alpha^{49}\rangle\} \quad
 (j\in\ZZ).
\]
For a fixed integer $a$ which is relatively prime to $7$, we put
\[
R_k=\bigcup_{i=0}^6 H_{a(7(k-1)+i)}.
\]
By computer, we have verified that the graph $\Gamma_k$ 
on $\GF(2^{21})$ with edge set $R_k$
is a strongly regular graph with eigenvalues 
$299593,585,-439$, for each $k\in\{1,\ldots,7\}$.
Together with the diagonal relation $R_0$,
we obtain a $7$-class pseudocyclic association 
scheme $(\GF(2^{21}),\{R_i\}_{i=0}^{7})$
satisfying the hypothesis of Theorem~\ref{main}.
The principal part of the first eigenmatrix is a linear combination
of an incidence matrix of $\PG(2,2)$ and the all-ones matrix.

By Theorem~\ref{thm:proj}(i), 
we obtain a non-amorphous $4$-class fusion scheme of the
cyclotomic scheme of class $49$ on $\GF(2^{21})$ 
with the following first eigenmatrix:
\begin{equation}  \label{e=49}
\begin{bmatrix}
1 & 1198372 & 299593 & 299593 & 299593 \\
1 & -1756 & 585 & 585 & 585 \\
1 & 292 & -439 & -439 & 585 \\
1 & 292 & -439 & 585 & -439 \\
1 & 292 & 585 & -439 & -439
\end{bmatrix}.
\end{equation}
This gives the third counterexample to A.V.~Ivanov's
conjecture having class $4$.
\end{example}

\paragraph{Acknowledgement.} The authors thank Qing Xiang
for bringing the reference \cite{SW} to the authors'
attention.

%\begin{remark}
%The main theorem holds even if the association
%scheme is imprimitive, that is, when each nontrivial
%relation is disconnected. 
%However, if we start with an imprimitive pseudocyclic
%association scheme satisfying the assumption of
%Theorem (i), then obviously the fused scheme of
%class $q+2$ is also imprimitive. Moreover, if
%an imprimitive pseudocyclic association scheme satisfies
%the assumption of Theorem (ii) then $\{r,s\}
%=\{f,-1\}$. Since $1+r k+s(d-k)$, we conclude
%$r=f$, and the fused scheme of
%class $d/(q+1)$ is also imprimitive. 
%\end{remark}

%%%%%%%%%%%%%%%%%%%%%%%%%%%%
%       References
%%%%%%%%%%%%%%%%%%%%%%%%%%%%

%\ifx\undefined\allcaps\def\allcaps#1{#1}\fi\newcommand{\nop}[1]{}
%\providecommand{\bysame}{\leavevmode\hbox to3em{\hrulefill}\thinspace}

\end{document}